\documentclass[12pt,
reqno]{amsart}%{article}
\usepackage{latexsym,amsmath,amssymb,amscd}
\usepackage{enumerate}
\usepackage{verbatim,color}
\begin{document}

\makeatletter
\@addtoreset{figure}{section}
\def\thefigure{\thesection.\@arabic\c@figure}
\def\fps@figure{h,t}
\@addtoreset{table}{bsection}

\def\thetable{\thesection.\@arabic\c@table}
\def\fps@table{h, t}
\@addtoreset{equation}{section}
\def\theequation{%\thesection.
\arabic{equation}}
\makeatother

\theoremstyle{plain}
\newtheorem{theorem}{Theorem}
\newtheorem{corollary}[theorem]{Corollary}
\newtheorem{definition}[theorem]{Definition}
\newtheorem{example}[theorem]{Example}
\newtheorem{lemma}[theorem]{Lemma}
\newtheorem{proposition}[theorem]{Proposition}
\newtheorem{question}[theorem]{Question}
\newtheorem{remark}[theorem]{Remark}

\numberwithin{theorem}{section}
\numberwithin{equation}{section}

%%% Todo
\newcommand{\todo}[1]{\vspace{5 mm}\par \noindent
\framebox{\begin{minipage}[c]{0.85 \textwidth}
\tt #1 \end{minipage}}\vspace{5 mm}\par}
%%%

%%%%ALL THE FOLLOWING \newcommand's MUST BE INCLUDED IN THE MAIN FILE

\newcommand{\1}{{\bf 1}}

\newcommand{\ad}{{\rm ad}}
\newcommand{\Ad}{{\rm Ad}}

\newcommand{\diag}{{\rm diag}}
\newcommand{\gl}{{{\mathfrak g}{\mathfrak l}}}
\newcommand{\id}{{\rm id}}
\newcommand{\Ker}{{\rm Ker}\,}
\newcommand{\Ran}{{\rm Ran}\,}
\newcommand{\rank}{{\rm rank}\,}
\newcommand{\slalg}{{\mathfrak{sl}}}
\renewcommand{\sp}{{{\mathfrak s}{\mathfrak p}}}
\newcommand{\so}{{{\mathfrak s}{\mathfrak o}}}

\newcommand{\Tr}{{\rm Tr}\,}

\newcommand{\CC}{{\mathbb C}}
\newcommand{\NN}{{\mathbb N}}
\newcommand{\RR}{{\mathbb R}}
\newcommand{\TT}{{\mathbb T}}
\newcommand{\ZZ}{{\mathbb Z}}

\newcommand{\bb}{b}%{{\mathpzc b}}
\newcommand{\vv}{v}%{{\mathpzc v}}
\newcommand{\ww}{w}%{{\mathpzc w}}

\newcommand{\Bc}{B}%{{\mathcal B}}
\newcommand{\Cc}{{\mathcal C}}
\newcommand{\Hc}{H}%{{\mathcal H}}
\renewcommand{\gg}{{\mathfrak g}}
\newcommand{\hg}{{\mathfrak h}}
\newcommand{\og}{{\mathfrak o}}
\newcommand{\ug}{{\mathfrak u}}
\newcommand{\nng}{{\mathfrak n}}
\newcommand{\pg}{{\mathfrak p}}
\newcommand{\Sginfty}{K}

\markboth{}{}

% Sample file: KWmyggams.tpl
% Typeset with LaTeX format

%\documentclass[12pt]{amsart}
%\usepackage{amssymb,latexsym}
%\usepackage[mathscr]{eucal}

% theorems, corollaries, lemmas, and propositions,
% in the most emphatic (plain) style;
% all are numbered separately
% There is a Main Theorem in the most emphatic (plain)
% style, unnumbered 
% There are definitions, in the less emphatic (definition) style
% There are notations, in the least emphatic (remark) style,
% unnumbered

%\theoremstyle{plain}
%\newtheorem{theorem}{Theorem}
%\newtheorem{corollary}[theorem]{Corollary}
%\newtheorem{lemma}[theorem]{Lemma}
%\newtheorem{proposition}[theorem]{Proposition}
%\newtheorem{remark}[theorem]{Remark}
%\newtheorem{definition}[theorem]{Definition}
%\newtheorem{example}[theorem]{Example}

%\DeclareMathOperator{\diag}{diag}

%\begin{document}

%\title{$B(H)$-Commutators: A Historical Survey II}
%\author{Gary Weiss}
%\email{gary.weiss@math.uc.edu}
%\address{University of Cincinnati\\
%         Department of Mathematics\\
%         Cincinnati, OH, 45221-0025\\
%         USA} 

\makeatletter
\title[$B(H)$-Commutators: A Historical Survey II and recent advances]
{$B(H)$-Commutators: A Historical Survey II  and 
recent advances on %\quad \\
commutators of compact operators}
\author[Daniel Belti\c t\u a]{Daniel Belti\c t\u a$^*$}
\thanks{$^*$%The research of Daniel Belti\c t\u a was 
Partially supported by the Grant
of the Romanian National Authority for Scientific Research, CNCS-UEFISCDI,
project number PN-II-ID-PCE-2011-3-0131, 
and by Project MTM2010-16679, DGI-FEDER, of the MCYT, Spain.}
\author[Sasmita Patnaik]{Sasmita Patnaik$^{**}$}
\thanks{$^{**}$Partially supported by the Graduate Dissertation Fellowship from the Charles Phelps Taft Research Center.}
\author[Gary Weiss]{Gary Weiss$^{***}$}
\thanks{$^{***}$%The research of Gary Weiss was 
Partially supported by 
Simons Foundation Collaboration Grant for Mathematicians \#245014.}
\address{Institute of Mathematics ``Simion
Stoilow'' of the Romanian Academy,
P.O. Box 1-764, Bucharest, Romania}
\email{Daniel.Beltita@imar.ro, beltita@gmail.com}
\address{University of Cincinnati, Department of Mathematics, Cincinnati, OH, 45221-0025, USA}
\email{sasmita\_19@yahoo.co.in}
\address{University of Cincinnati, Department of Mathematics, Cincinnati, OH, 45221-0025, USA}
\email{gary.weiss@math.uc.edu}
%\thanks{}
\dedicatory{Dedicated to the memory of Mih\'aly Bakonyi}
%\date{\today}

%\keywords{}
%\subjclass[2010]{Primary 46E22; Secondary 47B32, 46L05, 18A05, 58B12}
\makeatother


\begin{abstract}
A sequel to \cite{gW05}, we address again the single commutator problem \cite{PT71} of Pearcy and Topping: Is every compact operator a single commutator of compact operators? by focusing on a 35 year old test question for this posed in 1976 by the last named author and others: Are there any strictly positive operators that are single commutators of compact operators? The latter we settle here affirmatively with a modest modification of Anderson's fundamental construction \cite{jA77} constructing compact operators whose commutator is a rank one projection. Moreover we provide here a rich class of such strictly positive operators that are commutators of compact operators and pose a question for the rest. 

We explain also how these methods are related to the study of staircase matrix forms, their equivalent block tri-diagonal forms, and commutator problems. In particular, we present the original test question and solution that led to the negative solution of the Pearcy-Topping question on whether or not every trace class trace zero operator was a commutator (or linear combination of commutators) of Hilbert-Schmidt operators. And we show how this evolved from staircase form considerations along with a Larry Brown result on trace connections to ideals \cite{lB94} which itself is at the core of \cite[Section 7]{DFWW}. 

The omission in \cite{gW05} of this important 35 year old test question was inadvertent and we correct that in this paper.
This sequel starts where [ibid] left off but can be read independently of [ibid]. 

The present paper also has a section on self-commutator equations $[X^*,X]=A$ within the framework of certain classical Lie algebras of compact operators. That is, for the target operator $A$ in an operator Lie algebra one tries to find a solution $X$ in the same operator Lie algebra.
That problem was solved by P. Fan and C.K. Fong (1980) in the case of the full algebra of compact operators, and we establish versions of that result for the complex symplectic Lie algebra of compact operators as well as for any finite-dimensional  complex semisimple Lie algebra.
\end{abstract}

\maketitle

%Body of paper

\section{Introduction}\label{Sect1}
Commutators, linear operators of the form $AB-BA$, appear early on for instance in a mathematical formulation of Heisenberg's Uncertainty Principle \cite{wH27}. 
A simple concrete example is the 
product rule in calculus applied to $xf$  expressed in terms of operators: 
$I = \frac{d}{dx}M_x - M_x\frac{d}{dx}$ where the operators act on the class of differentiable functions. 
The situation changes in $B(H)$, that is, when the operators act boundedly on a Hilbert space. Wintner \cite{aW47} and Wielandt \cite{hW49} in 1947 and 1949, respectively, gave two elegant distinct proofs that the identity is not a commutator of two bounded linear operators on a Hilbert space. 
Both apply also to arbitrary complex normed algebras with unit, except Wintner's proof requires that the norm be complete. 
For the period preceding 1967, A Hilbert Space Problem Book 
\cite{pH82}-Chapter 24, provides a brief history of 
$B(H)$-commutators 
including some proofs. 

The definitive result on $B(H)$-commutators is due to Arlin Brown and Carl Pearcy \cite{BP65} (1965) characterizing its commutators as the non-thin operators,
where the thin operators are operators of the form $\lambda I + K$ with $0 \ne \lambda \in \mathbb C, K \in K(H)$, the ideal of compact operators.

The precursor to this paper \cite{gW05} starts with an elementary 
description of the subject similar to the viewpoint held by the author in the 1970's and continues with a report on the main contributions including references and some open problems spanning 1971-2003 from which our deeper understanding of the subject evolved. 
This subject of commutators of compact operators began with a series of questions due to Pearcy and Topping \cite{PT71} 
and its historical impact described at length in \cite{gW05}.
This sequel is intended to be independent of [ibid], but if more historical breadth should become of interest, at least its introduction should be consulted.
Other notable contributions to single commutators in the context of operator ideals, operator algebras, or minimizing commutator sum representations, came from \cite{lB94}, 
\cite[Section 7]{DFWW}, \cite{MarcouxSmallNumberCommutators}, \cite{MarcouxIrishSurvey}, \cite{Fack} 
(see also their substantial bibliographies). 

This paper reports also a new result: a positive solution to the 35 year old test question mentioned in the abstract exploiting the role in commutator theory of 
%staircase matrix forms and their closely related tri-block diagonal forms, 
tri-block diagonal forms and their closely related staircase forms,
in particular, we make a modest modification of Anderson's deep construction in this subject \cite{jA77}. This leads us herein to pose our next test question: \textit{Characterize in terms of eigenvalues (including multiplicities) which, if not all, strictly positive compact operators are commutators of compact operators}.

Acknowledgment. In the mid 2000's Ken Davidson communicated to the last named author that joint with Marcoux and Radjavi they had solved this test question affirmatively using Anderson's construction, but to date we have not seen this published. In our attempts to strengthen Anderson's construction to cover all strictly positive compact operators and absent that to find some, Davidson's communication played an invaluable motivation. 
We also wish to thank Karl-Hermann Neeb for drawing our attention to Corollary~\ref{Oberwolfach}. 

\section{In the beginning}
The commutator matrix constructions and their various norm formulas developed in \cite{gW75} and \cite{gW86} 
(also described in short in the survey \cite[Theorem 2.1 and Section 4, esp Problem 7]{gW05}) revealed the importance of 
focusing on the diagonal trace class matrices $\diag(-d,d_1,d_2,\dots)$ with $d_n \downarrow 0$ and $d := \sum_{1}^{\infty} d_n$ 
(and by simple normalization the special cases $1 = \sum_{1}^{\infty} d_n$) to determine which of these are commutators of Hilbert-Schmidt operators.
And from this, to focus on the special case finite matrix problem: \\

\noindent Compute the Hilbert-Schmidt norm minimum over $A \in M_4(\mathbb C)$
$$\text{min} \{\Vert A\Vert_{C_2} \mid AB-BA = 
\begin{pmatrix}
-1& 0& 0& 0\\
0& 1/3& 0& 0\\
0& 0& 1/3& 0\\
0& 0& 0& 1/3
\end{pmatrix}\}
%diag(-1,1/3,1/3,1/3)\}
$$
subject to scalar normalizing to insure $\Vert A\Vert_{C_2} = \Vert B\Vert_{C_2}$.\\
 
One has the trace norm/Hilbert-Schmidt norm inequality 
$$2\Vert A\Vert_{C_2}^2 = 2\Vert A\Vert_{C_2}\Vert B\Vert_{C_2} \ge \Vert AB\Vert_{C_1} + \Vert BA\Vert_{C_1} \ge \Vert AB-BA\Vert_{C_1} = 2$$
from which one sees that $\Vert A\Vert_{C_2}^2 \ge 1$, that is, the aforementioned minimum is at least $1$.
Indeed the same holds for the entire class $\diag(-d,d_1,d_2,\dots)$ after normalizing to $d = 1$, 
which made the focus on the class $\diag(-1,1/N,\dots,1/N)$ essential and subsequently led to the full solution of the Pearcy-Topping trace class trace zero problem.
The case $\diag(-1,1/2,1/2)$ nontrivally had minimum precisely $1$, and so the question sat from 1973--1976. 
Whether or not this minimum increased to infinity as $N$ increased to infinity at that time seemed essential to solving the whole problem.
The solution showcases the birth of staircase forms, at least for the author \cite{gW80}.

\begin{theorem}[Weiss, 1980]\label{T: 4/3}
$$\min \{\Vert A\Vert_{C_2} \mid AB-BA = \diag(-1,1/3,1/3,1/3)\} = \sqrt\frac{4}{3}.$$
The minimum is attained using \cite[Proposition 8.1]{gW75}:
$$A = \frac{1}{\sqrt 3}\begin{pmatrix}
0 &0 &0 &-1\\
\sqrt 2 &0 &0 &0\\
0 &1 &0 &0\\
0 &0 &0 &0
\end{pmatrix}
\quad \text{and} \quad 
B = \frac{1}{\sqrt 3}\begin{pmatrix}
0 &\sqrt 2 &0 &0\\
0 &0 &1 &0\\
0 &0 &0 &0\\
1 &0 &0 &0
\end{pmatrix}$$
\end{theorem}

\begin{proof}[Proof that $\sqrt \frac{4}{3}$ is a lower bound.]
See below.
\end{proof}

This led to the full solution of the Pearcy-Topping trace class trace zero problem (Theorem \ref{T: HSII}) by 
determining which among this somewhat general class of diagonal trace class operators are commutators of Hilbert-Schmidt operators.

\begin{theorem}[Weiss \cite{gW80}, 1980]\label{T: HSII}
Setting $d := \sum_{1}^{\infty} d_n$ for an arbitrary sequence $d_n \downarrow 0$, the following are equivalent.
\begin{enumerate}[{\rm(i)}]
\item%[(i)] 
$\diag(-d,d_1,d_2,\dots) \in [C_2,C_2]$
\item%[(ii)] 
$\diag(-d,d_1,d_2,\dots) \in [C_1,B(H)]$
\item%[(iii)] 
$\sum_{1}^{\infty} d_n \log n < \infty$. 
\end{enumerate}

In particular, if $\langle d_n\rangle = \langle\frac{1}{n\log^2 n}\rangle$, then 
\[
\diag(-d,d_1,d_2,\dots) \in C_1^o \setminus [C_2,C_2]. 
\]
\end{theorem}

And the totally general modern result is:
\begin{theorem}[Dykema, Figiel, Weiss and Wodzicki \cite{DFWW}, 2004]\label{T: DFWW}  \quad \\
If $I,J$ are two arbitrary $B(H)$-ideals, at least one of which is proper, \linebreak 
and $T = T^* \in IJ$, then
\[
T \in [I,J] ~\mbox{if and only if}~ \diag~ \lambda(T)_a \in IJ.
\] 
($\lambda(T)_a$ denotes the arithmetic mean sequence formed from the eigenvalue sequence of $T$, arranged in order of decreasing moduli, counting multiplicities
and when finite rank, ending in infinitely many zeros.)

Consequently, $[I,J] = [IJ,B(H)]$.
\end{theorem}

Here is the proof of the ``$\frac{4}{3}$" Theorem introducing also staircase forms.

\begin{proof}[Proof of Theorem \ref{T: 4/3}.] \quad \\
Assume 
\begin{equation}\label{T: 4/3 eq1}
AB-BA = \diag(-1,1/3,1/3,1/3) 
\end{equation}
which solutions exist since this finite matrix has trace $0$ 
(a more general result is due to K. Shoda-1937, for reference see \cite[Bibliography]{DFWW}), and normalized by scalar multiplication insures $\Vert A\Vert_{C_2} = \Vert B\Vert_{C_2}$. 
It is clear that the sequence $e_1, Ae_1, A^*e_1, e_2, e_3, e_4$ spans $\mathbb C^4$ 
($\{e_i\}_{i=1}^4$ denotes the standard basis) and that the Gram-Schmidt process yields another basis for $\mathbb C^4$. 
This provides an associated unitary $U$ that fixes $e_1$ and, for which~$U$, $\Ad_U$ leaves invariant $\diag(-1,1/3,1/3,1/3)$, 
that is, $U^*\diag(-1,1/3,1/3,1/3)U = \diag(-1,1/3,1/3,1/3)$ (equivalently, this diagonal remains the same under this basis change). 
And simultaneously this new basis puts $A$ into ``staircase" form:
$U^*AU = \begin{pmatrix}
*& *& *& 0\\
*& *& *& *\\
0& *& *& *\\
0& *& *& *
\end{pmatrix}$
because of how the Gram-Schmidt process works. That is, 
$A$ sends $e_1$ into a linear combination of the first 2 or fewer vectors in the new basis (depending on their linear independence),
and $A^*$ sends $e_1$ into a linear combination of the first 3 or fewer vectors in the new basis (depending on their linear independence).\\

Computing the diagonal entries of the commutator $AB-BA$ in terms of $A = (a_{ij})$ and $B = (b_{ij})$ one obtains the 4 equations:
\begin{align*}
-1 &= a_{12}b_{21} - b_{12}a_{21} + a_{13}b_{31} - b_{13}a_{31}\\
\frac{1}{3}  &= a_{21}b_{12} - b_{21}a_{12} + a_{23}b_{32} - b_{23}a_{32} + a_{24}b_{42} - b_{24}a_{42}\\
\frac{1}{3}  &= a_{31}b_{13} - b_{31}a_{13} + a_{32}b_{23} - b_{32}a_{23}+ a_{34}b_{43} - b_{34}a_{43}\\
\frac{1}{3} &= a_{42}b_{24} - b_{42}a_{24}+ a_{43}b_{34} - b_{43}a_{34}
\end{align*}
Summing the first 3 equations and taking the first equation yields the 2 equations:
\begin{align*}
-1 &= a_{12}b_{21} - b_{12}a_{21} + a_{13}b_{31} - b_{13}a_{31}\\
\frac{1}{3} &= a_{42}b_{24} - b_{42}a_{24}+ a_{43}b_{34} - b_{43}a_{34}
\end{align*}
It so happens that this second equation is the last of the previous 4 so the summing process is not necessary to obtain it. 
But it is this summing process that generalizes to prove Theorem \ref{T: HSII}.\\

Subtracting one has:
$$-\frac{4}{3} = a_{12}b_{21} - b_{12}a_{21} + a_{13}b_{31} - b_{13}a_{31} - (a_{42}b_{24} - b_{42}a_{24}+ a_{43}b_{34} - b_{43}a_{34}$$
and hence using the triangular and H\"{o}lder inequalities,
\begin{align*}
\frac{4}{3} &\le |a_{12}||b_{21}| + |b_{12}||a_{21}| + |a_{13}b_{31}| + |b_{13}||a_{31}| \\
&\qquad + |a_{42}||b_{24}| + |b_{42}||a_{24}| + |a_{43}||b_{34}| + |b_{43}||a_{34}|\\
&\le \sqrt {|a_{12}|^2 + |a_{21}|^2 + |a_{13}|^2 + |a_{31}|^2 + |a_{42}|^2 + |a_{24}|^2 + |a_{43}|^2 + |a_{34}|^2} \\
&\quad \times \sqrt {|b_{21}|^2 + |b_{12}|^2 + |b_{31}|^2 + |b_{13}|^2 + |b_{24}|^2 + |b_{42}|^2 + |b_{34}|^2 + |b_{43}|^2} \\
&\le \Vert A\Vert_{C_2}\Vert B\Vert_{C_2} = \Vert A\Vert^2_{C_2}.
\end{align*}
The last inequality arises from observing that each $a_{ij}, b_{ij}$ appears no more than once each in the first inequality, and some appear not at all.
The last equality follows from the assumed scalar normalization to make 
$\Vert A\Vert_{C_2} = \Vert B\Vert_{C_2}$ in the equation~\eqref{T: 4/3 eq1}.
%$AB-BA = \diag(-1,1/3,1/3,1/3)$.
Without this normalization one has in general that $\Vert A\Vert_{C_2}\Vert B\Vert_{C_2} \ge \frac{4}{3}$.
\end{proof}

The general staircase form result \cite[Corollary 3]{gW80} (modified here) that led to Theorem~\ref{T: HSII} (\cite[Theorem 5]{gW80}) is:

\begin{corollary}\label{C: HSII}
If $A_1,\dots,A_N$ denotes any finite collection of operators in $B(H)$, 
then there exists a unitary operator $U$ fixing $e_1$ so that $A_1,\dots,A_N$ transform simultaneously matrices with  
their $n^{th}$ row and column nonzero in at most the first $n(2N + 1)$ entries.
If they are selfadjoint, then they are thinner-as above but nonzero for at most $n(N + 1)$ entries.

For a single selfadjoint matrix, this form with inducing change of basis unitary is:
$$U^*AU = \begin{pmatrix}
&* &* &* &3 &0 &0 &0 &0 &\cdots\\
&* &* &* &* &* &* &6 &0 &\cdots\\
&* &* &* &* &* &* &* &* &\cdots\\
&3 &* &* &* &* &* &* &* &\cdots\\
&0 &* &* &* &* &* &* &* &\cdots\\
&0 &* &* &* &* &* &* &* &\cdots\\
&0 &6 &* &* &* &* &* &* &\cdots\\
&0 &0 &* &* &* &* &* &* &\cdots\\
&&&\vdots
\end{pmatrix}$$
\end{corollary}

In summary, Theorem \ref{T: 4/3} and Corollary \ref{C: HSII} together provided the necessary tools to obtain the full solution to the Pearcy-Topping trace class trace zero problem as mentioned in Theorem \ref{T: HSII}.\\

Notable also relating staircase forms (in particular, the equivalent block upper Hessenberg diagonal perspective) to commutators and subtle trace phenomena is Larry Brown's ideal result \cite{lB94}:

\begin{theorem}[L. G. Brown \cite{lB94}, 1994]\label{T: L.G.Brown} \quad \\
If $A\in C_p$, $B\in C_q$, $p^{-1}+q^{-1}\geq \frac{1}{2}$ and the commutator $[A,B]$ has finite rank, then $\Tr~[A,B]=0$.
\end{theorem}

Here the assumption that the commutator has finite rank leads to block upper Hessenberg diagonal forms with blocks growing arithmetically in size,
an essential feature to make his analytic estimates work.

\section{A rich class of strictly positive compact operators\\ that are single commutators of compact operators}

\begin{theorem} Positive compact operators are commutators of compact operators when they have eigenvalue sequences: 
$$(d_1, \frac{d_2-d_1}{2}, \frac{d_2-d_1}{2}, \frac{d_3-d_2}{3}, \frac{d_3-d_2}{3}, \frac{d_3-d_2}{3}, \dots)$$ 
where $0 \le d_n \uparrow$ but $\frac{d_n}{n} \rightarrow 0$.\\
\textbf{Constructing examples, particularly strictly positive ones, is easy:} 

$d_n = \sqrt n$ or $d_n = log\, n.$

An equivalent and more direct condition is: \\
%For $0< s_n \uparrow$ and $d_n := s_n - s_{n-1}$, $n \geq 2$ with $d_1 = s_1$, \\
if $d_n \ge 0$ and $\frac{1}{n}\displaystyle{\sum_{j=1}^n}d_j \rightarrow 0$, then the positive compact operators with eigenvalue sequence
$$(d_1, \frac{d_2}{2}, \frac{d_2}{2}, \frac{d_3}{3}, \frac{d_3}{3}, \frac{d_3}{3}, \cdots)$$
are single commutators of compact operators. 

\end{theorem}

\begin{proof} Modify via elementary means Anderson's construction \cite{jA77} for the rank one projection $P$ using his notation.
We prove here the first of the two equivalent conditions of the theorem.

The standard rank one projection
\[
  P=
  \left( \begin{array}{ccc}
	1      & 0  & \hdots \\
	0      & 0 & {}     \\
	\vdots & {} & \ddots \end{array}\right)
\]
admits the commutator representation $P = [C, Z]$ in terms of block
tri-diagonal matrices
\[
C =
\left( \begin{array}{cccc}
	0 & A_1\\
	B_1 & 0 & A_2\\
	{} & B_2 & 0 & \ddots\\
	{} & {} & \ddots & \ddots\end{array}\right)
  \quad\text{and}\quad
Z =
\left( \begin{array}{cccc}
	0 & X_1\\
	Y_{1} & 0 & X_2\\
	{} & Y_2 & 0 & \ddots\\
	{} & {} & \ddots & \ddots\end{array}\right)
\]
where $A_n$ and $X_{n}$ are the $n\times(n+1)$ matrices of norm $\frac{1}{\sqrt n}$
\[
A_n =
{\frac 1n} \left( \begin{array}{ccccc}
	\sqrt n &  0 \\
	{} & \sqrt {n-1} & 0  \\
	{} &   {} & \ddots & \ddots\\
	{} &   {} & {}     & \sqrt 1 &      0\end{array}\right)
  \quad\text{and}\quad
X_n =
{\frac 1n} \left( \begin{array}{ccccc}
	0 & \sqrt 1 \\
	{} & 0 & \sqrt 2 \\
	{} & {} & \ddots & \ddots & {}\\
	{} & {} & {} & 0 & \sqrt n\end{array}\right)
\]
while $B_n$ and $Y_n$ are the $(n+1)\times n$ matrices  of norm $\frac{\sqrt n}{n+1}$
\[
B_n =
-{\frac 1{n+1}} \left( \begin{array}{cccc}
	0 \\
	\sqrt 1 & 0 & \\
	{} & \sqrt 2 & \ddots & {} \\
	{} & {} &  \ddots & 0\\
	{} & {} &  {} & \sqrt n\end{array}\right) \,
 \quad\text{and}\quad
Y_n =
{\frac 1{n+1}} \left( \begin{array}{cccc}
	\sqrt n & {} & {} & {}\\
	0 & \sqrt {n-1} & {} & {}\\
	{} & 0  & \ddots & {}\\
	{} & {} & \ddots & \sqrt 1\\
	{} & {} & {} & 0\end{array}\right).
\]
Then 
\[[C, Z] = \begin{pmatrix}
D_1 & 0 & U_1 & 0 & \cdots\\
0 & D_2 & 0 & U_2 & 0 & \cdots \\
L_1 & 0 & D_3 & 0 & U_3 & 0 \cdots \\
\vdots &\ddots&\ddots& \ddots&\ddots&\ddots&\ddots\\
0 & \cdots & L_n & 0 & D_{n+1} & 0 & U_{n+1} & 0 & \cdots \\
\vdots \\
\end{pmatrix}\]
where L's, D's  and U's are
\[
\begin{matrix}
L & D & U \\
B_2X_1-X_2B_1 & A_1 Y_1-X_1 B_1 = 1 & A_1 X_2-X_1 A_2 \\
B_3X_2-X_3B_2 & B_1 X_1-Y_1 A_1 + A_2 Y_2-X_2 B_2 & A_2 X_3-X_2 A_3 \\
\vdots \\
B_{n+1}X_n-X_{n+1}B_n &  B_n X_n-Y_n A_n + A_{n+1}Y_{n+1}-X_{n+1}B_{n+1} & A_n X_{n+1}-X_n B_{n+1}\\
0 & -\frac{I_n}{n} \qquad \qquad + \qquad \frac{I_n}{n} & 0\\
\end{matrix}
\]

New idea: preserve as much as possible these equations. \\

Elementary modification-replace each $A_n, B_n, X_n, Y_n$ by multiplying each by $\sqrt d_n$.\\

Result:
\[
\begin{matrix}
L & D & U \\
B_2X_1-X_2B_1 & A_1 Y_1-X_1 B_1 = d_1 & A_1 X_2-X_1 A_2 \\
B_3X_2-X_3B_2 & B_1 X_1-Y_1 A_1 + A_2 Y_2-X_2 B_2 & A_2 X_3-X_2 A_3 \\
\vdots \\
B_{n+1}X_n-X_{n+1}B_n &  B_n X_n-Y_n A_n + A_{n+1}Y_{n+1}-X_{n+1}B_{n+1} & A_n X_{n+1}-X_n B_{n+1}\\
0 & -d_n\frac{I_n}{n} \qquad \qquad + \qquad d_{n+1}\frac{I_n}{n} = \frac{d_{n+1}-d_n}{n}I_n & 0\\
\end{matrix}
\]\\

\noindent Therefore to obtain $[C,Z] > 0$ it suffices to choose $d_n \uparrow$ strictly but 
also satisfying the conditions $\sqrt d_n \frac{1}{\sqrt n}$ and $\sqrt d_n \frac{\sqrt n}{n+1} \rightarrow 0$.
Both convergence conditions are equivalent to the single condition $ \frac{d_n}{n} \rightarrow 0$.
\end{proof}

It is interesting to observe the diagonal entry increasing multiplicities, which obstruction we have not yet seen how to overcome.\\

\textit{To integrate this contextually,} 
Corollary \ref{C: HSII} originated from the well-known fact that single selfadjoint operators with a cyclic vector are tridiagonalizable.
And from this corollary for finite collections of selfadjoint operators it follows that every operator (or finite collection of operators) 
is simultaneously finite block-tridiagonal with block sizes growing at most exponentially. Notice that Anderson's construction has block sizes growing arithmetically.
So there is plenty of room for research-on the impacts of varying sizes, like Larry Brown's result mentioned earlier, which exploited arithmetic growth of finite block sizes 
that was a consequence of a commutator being of finite rank.\\

Next natural question: Can $[C,Z] > 0$ with distinct eigenvalues? \\
That is, can a more elaborate modification of Anderson's hard construction achieve more?\\

\noindent Remaining open question: Which positive operators are in single commutators of compact operators?\\

\noindent There may be hope for this (even the general Pearcy-Topping problem: 

Which compact operators are commutators of compact operators). \\

\noindent From staircase forms studied in the late 1970's:\\

\noindent Every $B(H)$ operator is a tri block diagonalizable operator 
where the blocks are rectangles of sizes increasing no more than exponentially.\\

\noindent Is there a way to improve block size control?

\section{Self-commutators in some operator Lie algebras}

The main theme of this section
is that whenever we have to solve an operator equation $[X,Y]=A$ 
it is natural to hope that the symmetry properties of the target operator $A$ are shared by 
the solution operators $X$ and $Y$. 
More specifically, the symmetry properties of $A$ could be encoded by the assumption that it belongs 
to some operator Lie algebra $\gg$ and then one could try to find $X,Y\in\gg$ for which $[X,Y]=A$, 
which gives a certain Lie theoretic flavor for this section of this paper. 

Here we investigate solutions of the commutator equation $[X,Y] = A$ under 
the additional assumptions that $\gg$ is an involutive complex Lie algebra, 
$A=A^*\in\gg$, and $X=Y^*$, hence we will try to solve the so-called \emph{self-commutator equations} $[X^*,X]=A$. 
The main new result is Theorem~\ref{th_C}, and its proof needs only 
the spectral theory for compact operators. 
The second subsection is devoted to presenting some Lie theoretic results that actually motivated us 
to seek the operator theoretic facts obtained in the first subsection, 
namely a new observation (Proposition~\ref{finite_dim}) on self-commutators in  
complex semisimple Lie algebras.  

Since Subsection~\ref{subsect2} deals with finite-dimensional Lie algebras, which are after all matrix Lie algebras 
(see Ado's theorem in \cite[Appendix B.3]{Kn02}), 
while Subsection~\ref{subsect1} deals with operator Lie algebras 
which are infinite-dimensional Lie algebras 
and consist of linear transformations on infinite-dimensional Hilbert spaces,  
we have here an instance of the old principle that the linear algebra structure 
often provides motivation for operator theory. 
We preferred to place the motivation at the end for the only reason that this way of presenting the facts 
emphasizes that one can follow the proof of Theorem~\ref{th_C} without any knowledge of Lie algebras. 

\subsection{Self-commutators in complex classical Lie algebras of compact operators}\label{subsect1}
\hfill 
\\ As in Section~\ref{Sect1}, let  $\Hc$ be a separable infinite-dimensional complex Hilbert space 
with the ideal of compact operators denoted by $\Sginfty(\Hc)$. 
We will discuss operator equations $[X^*,X]=A$ 
where both the unknown operator~$X$ and the given operator $A$ belong to 
one of the Lie algebras of compact operators 
identified below. 

\begin{definition}[\cite{dlH72}]\label{complex_alg}
\normalfont 
Three classical Lie algebras of compact operators are defined as follows.  
\begin{itemize}
\item the \emph{complex classical Lie algebra of type} (A):
$$
\gl_\infty(\Hc):=\Sginfty(\Hc)$$ 
\item the \emph{complex classical Lie algebra of type} (B):
$$
\og_\infty(\Hc):=\{X\in\Sginfty(\Hc)\mid X=-JX^*J^{-1}\},$$ 
where $J\colon\Hc\to\Hc$ is a conjugation 
(i.e., $J$ is a conjugate-linear isometry satisfying  
 $J^2=\1$, where conjugate-linear means additive and 
$J(\alpha v)=\bar\alpha Jv$ for all $\alpha\in\mathbb C$ and $v\in\Hc$) 
\item the \emph{complex classical Lie algebra of type} (C): 
$$
\sp_\infty(\Hc):=\{X\in\Sginfty(\Hc)\mid X=-\widetilde{J}X^*\widetilde{J}^{-1}\},$$ 
where $\widetilde{J}\colon\Hc\to\Hc$ is an anti-conjugation 
(i.e., $\widetilde{J}$ is a conjugate-linear isometry satisfying 
$\widetilde{J}^2=-\1$)
\end{itemize}%
Note that all of the above operator classes $\gl_\infty(\Hc)$, $\og_\infty(\Hc)$, and $\sp_\infty(\Hc)$ 
are involutive complex Lie algebras of compact operators, 
in the sense that they are complex linear subspaces of $\Sginfty(\Hc)$ and are closed 
under the operator commutator and the involution 
given by the Hilbert space adjoint.  
Indeed, $[X,Y] = [-JX^*J^{-1},-JY^*J^{-1}] 
= J[X^*,Y^*]J^{-1} = -J[X,Y]^*J^{-1}$, 
and likewise for $\widetilde{J}$. 
\end{definition}

\begin{remark}\label{anticonj}
\normalfont 
Since the above definition involves conjugate-linear iso\-metries  
and on the other hand the (complex-)linear isometries in Hilbert spaces are sometimes 
described as operators that preserve the scalar product, 
we recall that such a description in terms of scalar products has to be slightly 
changed in the case of the conjugate-linear isometries (\cite[Appendix to Chapter I]{dlH72}). 
Nevertheless, both conjugate-linear isometries and (complex-) linear isometries 
can be described in a unified manner as norm-preserving operators. 

We now recall from \cite{Ba69}, \cite[Appendix to Chapter I]{dlH72} and \cite[Lemma 7.5.6]{HS84} 
a few basic properties of anti-conjugations. 
A complex Hilbert space $H$ admits an anti-conjugation 
if and only if its complex dimension is an even integer or is infinite. 
If $H$ is infinite-dimensional and $\{\bb_n\}_{n\in\ZZ\setminus\{0\}}$ is an orthonormal basis, 
then an anti-conjugation on $H$ can be defined if we set $\widetilde{J}\bb_n=-\bb_{-n}$ 
and $\widetilde{J}\bb_{-n}=\bb_n$
for $n\ge 1$, and then extending $\widetilde{J}$ to a conjugate-linear isometry. 
A similar construction works for finite-dimensional Hilbert spaces of even dimension. 
Conversely, if the Hilbert space $H$ is endowed with an anti-conjugation~$\widetilde{J}$, 
then by using the properties of $\widetilde{J}$ 
(in particular, $\langle\widetilde{J}v,\widetilde{J}w\rangle=\langle w,v\rangle$ for $v,w\in\Hc$) 
one obtains 
$\langle v,\widetilde{J}v\rangle=-\langle\widetilde{J}^2v,\widetilde{J}v\rangle=-\langle v,\widetilde{J}v\rangle$, 
hence $v\perp\widetilde{J}v$, for every $v\in\Hc$.  
By using that observation along with Zorn's lemma, 
one can then construct an orthonormal set ${\mathcal S}\subseteq H$ 
which is maximal with the property that for every $v\in{\mathcal S}$ we also have $\widetilde{J}v\in{\mathcal S}$. 
Maximality of ${\mathcal S}$ and isometricity of $\widetilde{J}$ 
can be used for proving that ${\mathcal S}^\perp=\{0\}$, 
and moreover the property $\widetilde{J}^2=-\1$ shows that ${\mathcal S}$ is a disjoint union of 2-element sets of the form 
$\{v,\widetilde{J}v\}$. 
Thus ${\mathcal S}$ is an orthonormal basis of $H$ which either is infinite
or is finite and contains even number of elements. 
(See \cite[Section 3]{Be09} for more general constructions of orthonormal bases 
associated to conjugations or to anti-conjugations on Hilbert spaces.)

Any two anti-conjugations on $H$ are unitarily equivalent to each other, and one unitary equivalence
that maps them to each other also defines an isomorphism between
the complex classical Lie algebras  of type (C) defined by means of those two anti-conjugations. 
In fact, let $\widetilde{J}_1$ and $\widetilde{J}_2$ be anti-conjugations on $H$ 
satisfying $\widetilde{J}_2=V\widetilde{J}_1V^*$ for some unitary 
operator $V$. 
If $X\in B(H)$ has the property $X=-\widetilde{J}_1X^*\widetilde{J}_1^{-1}$, 
then 
$$VXV^*=-(V\widetilde{J}_1V^*)(VX^*V^*)(V\widetilde{J}_1^{-1}V^*)=-\widetilde{J}_2(VXV^*)^*\widetilde{J}_2^{-1},$$ 
hence the unitary equivalence $\Ad_V\colon B(H)\to B(H)$, $X\mapsto VXV^*$, 
maps the complex classical Lie algebra of type (C) defined by means of $\widetilde{J}_1$ 
into the complex classical Lie algebra of type (C) defined by means of $\widetilde{J}_2$.
For this reason the anti-conjugation used for defining a complex classical Lie algebra of type (C)
is not reflected in the notation of that Lie algebra in this paper.

Similar remarks can be made about conjugations and complex classical Lie algebras  of type (B), 
except for the fact that their existence is not conditional on the even-dimensionality 
of the Hilbert space under consideration.  
That is, conjugations exist on every complex Hilbert space. 
\end{remark}

In the case of complex classical Lie algebras of type (A), 
the characterization of self-commutators of compact operators was obtained in \cite{FF80}. 
For convenience we introduce the following terminology 
in order to state that result (Theorem~\ref{Fan_th}). 

\begin{definition}\label{seq_A}
\normalfont
A sequence of real numbers $\langle\lambda_n\rangle_{n=1}^\infty$ 
is \emph{of type}~(A) 
if it satisfies the conditions 
$$\lim_{n\to\infty}\lambda_n=0\ \text{ and }\ 
\sum_{n\ge 1}\lambda_n^{+}=\sum_{n\ge 1}\lambda_n^{-}\  (\le\infty)$$
where for every $\lambda\in\RR$ we denote $\lambda^{\pm}:=(\vert\lambda\vert\pm\lambda)/2\ge0$. 
\end{definition}

\begin{example}\label{summable}
\normalfont
If in addition 
$\sum\limits_{n\ge1}\vert\lambda_n\vert<\infty$, 
then $\langle\lambda_n\rangle_{n=1}^\infty$ is a sequence of type~(A) if and only if $\sum\limits_{n\ge1}\lambda_n=0$. 
\end{example}

\begin{remark}\label{Fan_lemma}
\normalfont
With the above terminology one can state 
\cite[Lemma 2]{FF80} as: if $\langle\lambda_n\rangle_{n=1}^\infty$ is a sequence of type~(A), 
then there exists a permutation $\sigma\colon\{1,2,\dots\}\to\{1,2,\dots\}$ for which the sequence of sums 
$\langle\lambda_{\sigma(1)}+\cdots+\lambda_{\sigma(n)}\rangle_{n=1}^\infty$ 
is nonnegative and converges to~$0$. 
\end{remark}

Some motivation for introducing the terminology in Definition~\ref{seq_A} 
is provided by Question~\ref{question_BC} below, 
which basically asks for identifying the sequences of real numbers 
which can occur as eigenvalue sequences for self-commutators 
in the complex classical Lie algebras of type (A), (B), (C). 
It would be natural to call these sequences of types (A), (B), (C), respectively, 
but such a definition has the drawback that it is very implicit. 
For the  sequence of type (A) 
%and (C) 
we prefer a different approach in 
Definition \ref{seq_A}, 
%~and~\ref{seq_C}, 
and we further clarify our reasons as follows. 

The problem of identifying the eigenvalue sequences for Lie algebras of type (A) was solved in \cite{FF80} 
(see Theorem~\ref{Fan_th} below). 
For algebras of type (C), this problem will be solved in Theorem~\ref{th_C} below.  
%If the converse of Theorem~\ref{th_C} were true, 
%then the problem for classical Lie algebras of type (C) would be solved as well, 
%in the sense that the sequences introduced in Definition~\ref{seq_C} below would be the ones we are looking for. 
The problem for classical Lie algebras of type (B) is completely open. 
We do not address it in this paper and we will not propose any explicit definition for the sequences of type~(B).   

First we give the spectral characterization of self-commutators in classical Lie algebras of type (A), 
that is, the involutive Lie algebra 
of all compact operators, which resolved \cite[Problem III]{Pe72} for Fan-Fong \cite[Theorem 1]{FF80}.

\begin{theorem}\label{Fan_th}
If $T=T^*\in\gl_\infty(\Hc)=K(H)$, then the equation $T=[Y^*,Y]$ can be solved for $Y\in\gl_{\infty}(\Hc)$ 
if and only if the sequence of eigenvalues of $T$ repeated according to multiplicities is a sequence of type~(A). 
\end{theorem}

\begin{proof}[Idea of proof] 
See \cite[Theorem 1]{FF80} for details. 
If the sequence of eigenvalues of $T$ is of type~(A), 
then by using the above Remark~\ref{Fan_lemma},  
we may assume that this sequence has nonnegative initial partial sums. 
Then an operator $Y\in\gl_{\infty}(\Hc)$ satisfying $T=[Y^*,Y]$ can be constructed 
as the weighted shift operator whose weights are the square roots of the aforementioned partial sums. 
In other words, $Y$ is an operator defined by an infinite matrix version 
of the matrix~\eqref{Fan_ex_eq1} used in Example~\ref{Fan_ex} below. 
Thus, in \cite{FF80} the method of proof of this theorem 
is merely an extension of the method of Example~\ref{Fan_ex} 
to infinite dimensions which uses the above remark, 
and this stands in contrast to our approach employing Lemma~\ref{Fang_fin} below.
\end{proof}

We will specialize Theorem~\ref{Fan_th} to operators on finite-dimensional Hilbert spaces, 
and will provide full details of the proof in that situation. 
To this end we need the following elementary observation 
which is a simple finite-dimensional version of the fact recorded in Remark~\ref{Fan_lemma}. 

\begin{lemma}\label{Fang_fin}
Let $c_1,\dots,c_n\in\RR$ with $c_1+\cdots+c_n=0$. 
If $c_1\ge\cdots\ge c_n$, then 
for every $j=1,\dots,n$ we have $c_1+\cdots+c_j\ge 0$. 
\end{lemma}

\begin{proof}
Assume $c_1+\cdots+c_j<0$ for some $j\in\{1,\dots,n\}$. 
Then $\min\{c_1,\dots,c_j\}<0$, 
hence by using the hypothesis we obtain $0>c_j\ge c_{j+1}\ge\cdots\ge c_n$. 
It then follows that 
$c_1+\cdots+c_j+c_{j+1}+\cdots+c_n\le c_1+\cdots+c_j<0$, 
and this is a contradiction to the hypothesis $c_1+\cdots+c_n=0$. 
\end{proof}

\begin{example}\label{Fan_ex}
\normalfont 
Let $T=T^*\in M_{r+1}(\mathbb C)$. 
For finding a matrix $Y\in M_{r+1}(\mathbb C)$ with $[Y^*,Y]=T$, we need to assume $\Tr T=0$, 
which is the natural finite-dimensional version of the condition from Theorem~\ref{Fan_th} 
(compare Example~\ref{summable}). 
By using the spectral theorem we may assume that $T$ is a diagonal matrix, 
say $T=\diag(c_1,\dots,c_{r+1})$, where $c_1,\dots,c_{r+1}\in\RR$ and $c_1+\cdots+c_{r+1}=0$. 
Then after a suitable permutation of the vectors in $\mathbb C^{r+1}$ 
we may assume $c_1\ge\cdots\ge c_{r+1}$, 
since unitary equivalence of $T$ preserves the solvability of the self-commutator equation, 
and then by Lemma~\ref{Fang_fin} we have $a_j:=c_1+\cdots+c_j\ge 0$ for 
$j=1,\dots,r+1$. 
If we define 
\begin{equation}\label{Fan_ex_eq1}
Y:=%\sqrt{a_1}E_{12}+\cdots+\sqrt{a_r}E_{r,r+1}=
\begin{pmatrix}
0 & \sqrt{a_1} &  & \textbf{\large 0}\\
 &    \ddots    &          \ddots &   \\
 &        & \ddots &  \sqrt{a_r} \\
\textbf{\large 0} &        &  &  0
\end{pmatrix} 
\end{equation}
then we have 
$$\begin{aligned}
{}
[Y^*,Y]
&=\diag(a_1,a_2-a_1,a_3-a_2,\dots,a_{r+1}-a_r) \\
&=\diag(c_1,\dots,c_{r+1})=T
\end{aligned}$$
by a direct computation. 
See Example~\ref{finite_dim_ex} for a higher perspective on that computation. 
\end{example}

\begin{remark}
\normalfont
In connection with Example~\ref{Fan_ex} 
it is instructive to study the case $T=\diag(c_1,\dots,c_{r+1})$ with $c_1=\cdots=c_r=\frac{1}{r}$ 
and $c_{r+1}=-1$. 
In this case we have $a_j=\frac{j}{r}$ for $j=1,\dots,r$, 
hence the Hilbert-Schmidt norm of the matrix $Y$ is 
$\Vert Y\Vert_{C_2}=\Vert Y^*\Vert_{C_2}=\sqrt{a_1+\cdots+a_r}=\sqrt{\frac{r+1}{2}}$. 
In particular, for $r=3$ we have $\Vert Y\Vert_{C_2}=\sqrt{2}$, which is  larger than 
the minimum $\sqrt{\frac{4}{3}}$ provided by Theorem~\ref{T: 4/3}. 
\end{remark}

\begin{question}\label{question_BC}
\normalfont
Is it possible to establish versions of Theorem~\ref{Fan_th} 
for the complex classical Lie algebras  of types (B) and (C) 
introduced in Definition~\ref{complex_alg}? 
\end{question}

We will give a 
%partial 
complete answer to the above question in the case of Lie algebras of type~(C). 
To this end we need the following auxiliary result, which is essentially known, 
but we give it here for completeness. 
%To this end it is convenient to introduce some extra terminology. 
%
%\begin{definition}\label{seq_C}
%\normalfont
%We say that a sequence of real numbers $\langle\lambda_n\rangle_{n=1}^\infty$ 
%is a \emph{sequence of type}~(C) 
%if there exists a sequence $\langle\varepsilon_n\rangle_{n=1}^\infty$ in the set $\{\pm1\}$ 
%such that $\langle\epsilon_n\lambda_n\rangle_{n=1}^\infty$ is a sequence of type~(A). 
%\end{definition}
%
%\begin{example}
%\normalfont
%It is clear that every sequence of type (A) is also of type (C). 
%Positive sequences of type (C), hence which are not of type~(A), can be easily constructed. 
%For instance, if a nonincreasing sequence of positive real numbers 
%$\langle\lambda_n\rangle_{n=1}^\infty$ has the property that 
%$\lim\limits_{n\to\infty}\lambda_n=0$ 
%and every term in the sequence occurs with even multiplicity, 
%then the choice of alternating signs $\varepsilon_n=(-1)^n$ for $n=1,2,\dots$
%shows that $\langle\lambda_n\rangle_{n=1}^\infty$ is a sequence of type~(C). 
%\end{example}

\begin{lemma}\label{spec_C}
Let $\widetilde{J}\colon\Hc\to\Hc$ be the anti-conjugation involved in the definition of $\sp_\infty(\Hc)$. 
If $T=T^*\in\sp_\infty(\Hc)$, then the following assertions hold: 
\begin{enumerate}[{\rm (i)}]
\item\label{spec_C_item1} 
For every $\lambda\in\RR$ we have the conjugate-linear isometry 
$$\widetilde{J}\vert_{\Ker(T-\lambda)}\colon\Ker(T-\lambda)\to\Ker(T+\lambda)$$ 
and $\dim\Ker(T-\lambda)=\dim\Ker(T+\lambda)$. 
Moreover $\dim\Ker T$ is either an even natural number or infinite.
\item\label{spec_C_item2} 
%If $T$ has the property that the sequence of eigenvalues of its positive part $T^{+}:=(\vert T\vert+T)/2$ 
%repeated according to multiplicities is a sequence of type~(C), 
%then the nonzero 
The eigenvalues of $T$ can be labeled as $(\lambda_1,-\lambda_1,\lambda_2,-\lambda_2,\dots)$ 
where 
%the positive elements is the sequence 
%$\langle\vert\lambda_n\vert\rangle_{n=1}^\infty$ which is the sequence of nonzero eigenvalues of $T^{+}$ 
%repeated according to multiplicities, 
%such that $\langle\lambda_1+\cdots+\lambda_n\rangle_{n=1}^\infty$ 
$\langle\lambda_n\rangle_{n=1}^\infty$ is a sequence of nonnegative real numbers converging to~$0$. 
\end{enumerate}
\end{lemma}

\begin{proof}
\eqref{spec_C_item1} 
Since $T=T^*\in\sp_\infty(\Hc)$ we have $T=-\widetilde{J}T^*\widetilde{J}^{-1}=-\widetilde{J}T\widetilde{J}^{-1}$, 
hence $T\widetilde{J}=-\widetilde{J}T$. 
Then for all $\lambda\in\RR$ and $v\in\Hc$  with $Tv=\lambda v$ we have 
$T\widetilde{J}v=-\widetilde{J}Tv=-\lambda\widetilde{J}v$. 
This shows that $\widetilde{J}$ maps $\Ker(T-\lambda)$ into $\Ker(T+\lambda)$, 
and this conjugate-linear isometry is surjective, since $\widetilde{J}^2=-\1$, hence $\widetilde{J}^{-1}=-\widetilde{J}$. 
For $\lambda=0$ it follows that $\widetilde{J}\vert_{\Ker T}$ is an anti-conjugation 
on the complex Hilbert space $\Ker T$, 
and then the assertion on the dimension of $\Ker T$ follows by 
what was already mentioned in the second paragraph of Remark~\ref{anticonj}. 

\eqref{spec_C_item2} 
It follows from Assertion~\eqref{spec_C_item1} that if $\lambda$ is a nonzero eigenvalue of $T$, 
then also $-\lambda$ is a nonzero eigenvalue of $T$ and moreover $\lambda$ and $-\lambda$ 
have equal spectral multiplicities. 
Therefore the sequence of eigenvalues of T repeated according to multiplicities including zero 
can be labeled as indicated in the statement. 
%(if it is an eigenvalue) is a type (A) sequence.
%Now the assertion is achieved by applying 
%sequentially Definition~\ref{seq_C}, Definition~\ref{seq_A} and 
%Remark~\ref{Fan_lemma}. 
%To this end start with the sequence of positive eigenvalues of $T$, that is, the eigenvalue sequence of $T^+$. 
%In fact, let us denote it by $\langle\mu_1, \mu_2,\dots\rangle$.
%Now since this is a type~(C) sequence, therefore there is a sequence 
%$\langle\varepsilon_n\rangle$ where $\varepsilon_n = 1$ or $-1$ such that the sequence
%$\langle\varepsilon_n \mu_n\rangle$ is a type~(A) sequence.
%Therefore there is a rearrangement of the sequence  $\langle\varepsilon_n \mu_n\rangle$ namely
% $\langle\varepsilon_{n_1}\mu_{n_1}, \varepsilon_{n_2}\mu_{n_2},\dots\rangle$ 
% such that the partial sum sequence
%$\langle\varepsilon_{n_1}\mu_{n_1} + \varepsilon_{n_2}\mu_{n_2} +\cdots+ \varepsilon_{n_k}\mu_{n_k}\rangle$ 
%is nonnegative and converges to zero.
%Then redefine  $\lambda_1 := \varepsilon_{n_1}\mu_{n_1}$, $\lambda_2 := \varepsilon_{n_2}\mu_{n_2}$ 
%and so on.
%And then one has under the above notation, the sequence
%$\langle\lambda_1 + \lambda_2 +\cdots+ \lambda_k\rangle$ converges to zero and is nonnegative.
%Then construct the sequence $\langle\lambda_1, -\lambda_1, \lambda_2, -\lambda_2,\dots\rangle$, 
%and this concludes the proof.
\end{proof}

We are now ready to prove the main new result of this section 
which provides a type~(C) version of Theorem~\ref{Fan_th}. 
The proof of the following theorem is self-contained except for Lemma~\ref{spec_C}. 
See also Remark~\ref{secret}  below for an explanation on the method of proof used here. 
%Recalling Definitions \ref{seq_A} and \ref{seq_C} 
%which together define type (C) sequences, we have the following theorem.

\begin{theorem}\label{th_C}
%Let $T=T^*\in\sp_\infty(\Hc)$ with positive part $T^{+}$. 
%If the sequence of eigenvalues of $T^{+}$ repeated according to multiplicities is a sequence of type~(C), 
If $T=T^*\in\sp_\infty(\Hc)$, then the equation $T=[Y^*,Y]$ can be solved for $Y\in\sp_{\infty}(\Hc)$. 
\end{theorem}

\begin{proof}
%Since $\Ker T$ is a reducing subspace for $T$ and the operator $0$ 
%can be trivially represented as a self-commutator $[Y_0^*,Y_0]$, 
%it easily follows by Lemma~\ref{spec_C}\eqref{spec_C_item1} that 
%restricting $T$ to $(\Ker T)^\perp$ this restriction retains the eigenvalue constraints of the hypotheses 
%and so we may assume $\Ker T=\{0\}$. 
%Then we 
We can use Lemma~\ref{spec_C}\eqref{spec_C_item2} for labeling 
the eigenvalues of $T$ as $(\lambda_1,-\lambda_1,\lambda_2,-\lambda_2,\dots)$, 
where $\lambda_n\ge 0$ 
%$\lambda_1+\cdots+\lambda_n\ge0$ 
for every $n\ge 1$ 
and 
%$\lim\limits_{n\to\infty}(\lambda_1+\cdots+\lambda_n)=0$
$\lim\limits_{n\to\infty}\lambda_n=0$. 
Let us pick an orthonormal sequence $\{\bb_n\mid n\ge 1\}$ 
such that $T\bb_n=\lambda_n\bb_n$ for every $n\ge 1$. 
If we define $\bb_{-n}:=-\widetilde{J}\bb_n\in\Ker(T+\lambda_n)$ for $n\ge 1$ 
(see Lemma~\ref{spec_C}\eqref{spec_C_item1} again), 
then $\bb:=\{\bb_n\mid n\in\ZZ^*:=\ZZ\setminus\{0\}\}$ is an orthonormal basis of~$\Hc$, 
and throughout the proof it will be convenient to use matrices of the operators on $\Hc$ 
with respect to that basis. 
For $j,k\in\ZZ^*$ we introduce the rank-one operator $E_{jk}:=(\cdot,\bb_k)\bb_j\in\Bc(\Hc)$, 
which corresponds to the matrix that has all entries equal to $0$ 
except for the entry~$1$ on the $j$-th row and $k$-th column. 
Hence $E_{jk}^*=E_{kj}$, $E_{jk}E_{k\ell}=E_{j\ell}$, and $E_{jk} E_{q\ell} = 0$ for all $j,k,q,\ell\in\ZZ^*$ 
for which $k\ne q$. 

We have $T\bb_{\pm j}=\pm\lambda_j\bb_{\pm j}$ for every $j\ge 1$, 
hence $T$ is given by a diagonal matrix with respect to the basis $\bb$ and more precisely we have 
$$T=\sum_{j\ge 1}\lambda_j(\underbrace{E_{jj}-E_{-j,-j}}_{\textstyle\hskip25pt=:H_j})$$
where the series is norm convergent in~$\Bc(\Hc)$. 
In order to construct a solution in $\sp_\infty(\Hc)$ for the self-commutator equation $[Y^*,Y]=T$ 
we first define 
%$H_j:=E_{jj}-E_{-j,-j}$ and 
$$X_j:=E_{-j,j}\text{ for }j\ge 1. 
%\begin{cases}
%E_{1,-1} &\text{ if }j=1, \\
%E_{-j,-(j+1)}-E_{j+1,j}&\text{ if }j\ge 2. 
%\end{cases}
$$
Note that 
%$[X_1^*,X_1]=[E_{-1,1},E_{1,-1}]=E_{-1,-1}-E_{1,1}=-H_1$ 
$[X_j^*,X_j]=[E_{j,-j},E_{-j,j}]=E_{jj}-E_{-j,-j}=H_j$. 
%and also for $j\ge 2$ we have 
%$$\begin{aligned}
%{} [X_j^*,X_j]
%&=[E_{-(j+1),-j}-E_{j,j+1},E_{-j,-(j+1)}-E_{j+1,j}] \\
%&=[E_{-(j+1),-j},E_{-j,-(j+1)}]+[E_{j,j+1},E_{j+1,j}]\\
%&=E_{-(j+1),-(j+1)} -E_{-j,-j} + E_{jj} - E_{j+1,j+1}\\
%&=H_j-H_{j+1}.
%\end{aligned}$$
Now define  
%$a_n:=\lambda_1+\cdots+\lambda_n$ for every $n\ge 1$,  and 
\begin{equation}\label{th_C_proof_eq1}
Y:=\sum_{n\ge 1}\sqrt{\lambda_n}X_n.
%Y:=\sum_{n\ge 1}\sqrt{a_n}X_n.
\end{equation}
In order to see that the above formula makes sense and $Y\in\sp_\infty(\Hc)$, 
%note that $\langle a_n\rangle_{n=1}^\infty$ 
recall that $\langle \lambda_n\rangle_{n=1}^\infty$ 
is a sequence of nonnegative numbers converging to $0$, 
according to the way the $\lambda_n$'s were defined at the beginning of the proof. 
On the other hand, relative to the basis $\{\bb_n\}$, the matrix representation of the operator $Y$ is bi-infinite and anti-diagonal 
%equal to $\sqrt{a_1}E_{1,-1}$ plus 
%the difference of a backward weighted shift and a forward weighted shift 
%in fact a backward weighted shift 
(see the definition of the $X_j$'s above)
whose weight sequence is convergent to~$0$.   
%The weight sequences of these weighted shifts are convergent to~$0$, hence both of them are compact operators. 
Thus $Y$ is in turn a compact operator. 
For proving that $Y\in\sp_\infty(\Hc)$, we still have to check that $Y=-\widetilde{J}Y^*\widetilde{J}^{-1}$, 
and to this end it suffices to check that $X_n=-\widetilde{J}X_n^*\widetilde{J}^{-1}$ for every $n\ge 1$. 
In fact, since $\widetilde{J}$ is conjugate-linear isometry and for $n\ge 1$ 
we have $\widetilde{J}\bb_{\pm n}=\mp \bb_{\mp n}$, 
it follows that for all $m,n\ge 1$ and every $v\in\Hc$ we have 
$$\begin{aligned}
\widetilde{J}E_{\mp m,\pm n}\widetilde{J}^{-1}v
&=-\widetilde{J}E_{\mp m,\pm n}\widetilde{J}v
=-\widetilde{J}((\widetilde{J}v,\bb_{\pm n})\bb_{\mp m}) \\
&=\widetilde{J}((\widetilde{J}v,\widetilde{J}^2\bb_{\pm n})\bb_{\mp m})
=\widetilde{J}((\widetilde{J}\bb_{\pm n},v)\bb_{\mp m}) \\
&=(v,\widetilde{J}\bb_{\pm n})\widetilde{J}\bb_{\mp m}=-(v,\bb_{\mp n})\bb_{\pm m}\\ 
&=-E_{\pm m,\mp n}v
=-E_{\mp n,\pm m}^*v.
%\widetilde{J}E_{\pm m,\pm n}\widetilde{J}^{-1}v
%&=-\widetilde{J}E_{\pm m,\pm n}\widetilde{J}v
%=-\widetilde{J}((\widetilde{J}v,\bb_{\pm n})\bb_{\pm m}) \\
%&=\widetilde{J}((\widetilde{J}v,\widetilde{J}^2\bb_{\pm n})\bb_{\pm m})
%=\widetilde{J}((\widetilde{J}\bb_{\pm n},v)\bb_{\pm m}) \\
%&=(v,\widetilde{J}\bb_{\pm n})\widetilde{J}\bb_{\pm m}=(v,\bb_{\mp n})\bb_{\mp m}\\ 
%&=E_{\mp m,\mp n}v.
\end{aligned}$$
Hence 
$\widetilde{J}E_{\mp n,\pm n}\widetilde{J}^{-1}=-E_{\mp n,\pm n}^*$,
%$\widetilde{J}E_{\pm m,\pm n}\widetilde{J}^{-1}=E_{\mp m,\mp n}$, 
where the subscripts are assumed to have 
%the same sign, 
opposite signs, 
and this implies $\widetilde{J}X_n^*\widetilde{J}^{-1}=-X_n$ for 
$n\ge 1$.
%$n\ge 2$. 
%The latter equality also holds true for $n=1$, 
%since a similar computation leads to $\widetilde{J}E_{\pm m,\mp n}\widetilde{J}^{-1}=-E_{\mp m,\pm n}$, 
%where the subscripts have opposite signs, for $m,n\ge 1$. 
This completes the verification of the fact that $Y\in\sp_\infty(\Hc)$. 
 
If $j\ne k$, 
%are larger than $2$
then $[X_j^*,X_k]=[E_{j,-j},E_{-k,k}]=0$.  
%$$[X_j^*,X_k]=[E_{-(j+1),-j}-E_{j,j+1},E_{-k,-(k+1)}-E_{k+1,k}]=0 $$  
%and it is clear that actually $[X_j^*,X_k]=0$ for all $j,k\ge 1$ with $j\ne k$. 
Therefore
$$%\begin{aligned}
%{}
[Y^*,Y]
%&=\sum_{n\ge 1}a_n[X_n^*,X_n] \\
%&=-a_1 H_1+a_2(H_1-H_2)+a_3(H_2-H_3)+\cdots \\
%&=(a_2-a_1)H_1+(a_3-a_2)H_2+\cdots \\
=\sum_{n\ge 1}\lambda_n[X_n^*,X_n] 
=\sum_{n\ge 1}\lambda_n H_n
=T
%\end{aligned}
$$
and this concludes the proof. 
\end{proof}

\begin{corollary}\label{cor_C}
For every $T\in\sp_\infty(\Hc)$, $T=[X^*,X]+ i[Y^*,Y]$ for some $X,Y \in\gg$. 
\end{corollary}

\begin{proof}
Write $T=T_1+iT_2$ with $T_j^*=T_j\in\sp_\infty(\Hc)$ for $j=1,2$, 
and apply Theorem~\ref{th_C}.  
\end{proof}

%\begin{question}\label{question_rec_C}
%\normalfont
%Does the converse to Theorem~\ref{th_C} hold true? 
%\end{question}

\begin{remark}
\normalfont
We note that an interesting result on representation of certain compact operators (with special symmetry properties) 
as single commutators was recently obtained in \cite[Prop. 5.4]{BP12}, 
but the corresponding construction is completely different from the ones used above. 
\end{remark}

\subsection{Single commutators in finite-dimensional complex Lie algebras}\label{subsect2} 
\hfill 
\\
%This subsection has 
A stronger Lie theoretic flavor can be noticed in this subsection 
as compared with the previous one 
and much more than the other sections of this paper, 
in the sense that 
the proof of its main new result (Proposition~\ref{finite_dim}) 
draws on the basic structure theory of complex semisimple Lie algebras. 
The purpose of the whole discussion below 
is to record a few results on single commutators in complex semisimple (finite-dimensional) Lie algebras  
that may suggest the results one could aim for in infinite dimensions.   
We first recall the following old result: 

\begin{theorem}\label{Gordon}
If $\gg$ is a finite-dimensional complex classical Lie algebra,  
then for every $A\in\gg$ there exist $X,Y\in\gg$ such that $[X,Y]=A$. 
\end{theorem}

\begin{proof}
See \cite{Br63} and also \cite{Hi90}.
\end{proof} 

If $\gg$ is a classical finite-dimensional Lie algebra of type (A), that is, 
if $\gg$ is the Lie algebra $\slalg(n,\mathbb C)$ of matrices in $M_n(\mathbb C)$ having the trace equal to zero, 
then one recovers the celebrated Shoda's theorem. 
Besides this special case, we recall that there exist three other 
%infinite series of 
classical Lie algebras of types (B), (C), and (D), namely $\og(2n+1,\mathbb C)$, 
$\sp(n,\mathbb C)$, and $\og(2n,\mathbb C)$ respectively, for $n=1,2,\dots$ 
(see \cite{Kn02}) 
to which Theorem~\ref{Gordon} also applies. 

%\section{Self-commutators in complex semisimple Lie algebras}

For the statement of the following result (Proposition~\ref{finite_dim}) 
we recall that a finite-dimensional complex Lie algebra $\gg$ is semisimple if 
and only if the so-called Killing form 
$$B_{\gg}\colon\gg\times\gg\to\mathbb C,\quad B_{\gg}(X,Y):=\Tr((\ad_{\gg}X)(\ad_{\gg}Y))$$
is nondegenerate, where for every $X\in\gg$ we define $\ad_{\gg}X\colon\gg\to\gg$, $(\ad_{\gg}X)Y=[X,Y]$ for $X,Y\in\gg$.  
See \cite[Theorem 1.45]{Kn02} for a proof of that characterization of semisimple Lie algebras, 
but we emphasize that for the purpose of understanding the statement of Proposition~\ref{finite_dim} below, 
the above characterization could be well used as a definition 
since it does not involve anything beyond the mere notion of a finite-dimensional Lie algebra.  
If $\gg$ is a complex semisimple Lie algebra, then there exists a conjugate-linear mapping $\gg\to\gg$, $X\mapsto X^*$, 
such that for all $X,Y\in\gg$ we have $(X^*)^*=X$, $[X,Y]^*=[Y^*,X^*]$ and $B_{\gg}(X,X^*)\ge 0$. 
Such a mapping will be called here a \emph{Cartan involution} on the complex semisimple Lie algebra~$\gg$ 
(although that name is usually reserved for the mapping $X\mapsto -X^*$), 
and it is unique up to an automorphism of~$\gg$. 
Now we can prove the following fact, which can be regarded as a self-adjoint version of Theorem~\ref{Gordon}.

\begin{proposition}\label{finite_dim}
Let $\gg$ be a finite-dimensional complex semisimple Lie algebra 
with a fixed Cartan involution.
Then for every $A\in\gg$ satisfying the condition $A=A^*$ 
there exists $Y\in\gg$ such that $A=[Y^*,Y]$.  
\end{proposition}

\begin{proof}
Since $A=A^*$, we can find a maximal abelian self-adjoint subalgebra $\hg$ of $\gg$ with $A\in\hg$. 
Denote the eigenspaces for the adjoint action of $\hg$ on $\gg$ by 
$$\gg^\alpha:=\{X\in\gg\mid(\forall H\in\hg)\quad [H,X]=\alpha(H)X\}
\text{ for a linear functional }\alpha\colon\hg\to\mathbb C$$
and consider the corresponding set of roots
$$\Delta(\gg,\hg):=\{\alpha\in\hg^*\mid \gg^\alpha\ne\{0\}\}. $$
Then $\hg=\gg^0=\{X\in\gg\mid[\hg,X]=\{0\}\}$ and the root space decomposition is given by 
$$\gg=\hg\oplus\bigoplus_{\alpha\in\Delta(\gg,\hg)\setminus\{0\}}\gg^\alpha$$
by \cite[Equations (2.16) and (2.22)]{Kn02}. 
For every $\alpha\in\Delta(\gg,\hg)\setminus\{0\}$ we have $\dim\gg^\alpha=1$ 
and we may choose $X_\alpha\in\gg^\alpha\setminus\{0\}$ such that $X_\alpha^*=X_{-\alpha}$ 
(see \cite[Chapter VI, \S 1--2]{Kn02}). 
Let $\{\alpha_1,\dots,\alpha_r\}\subset\Delta(\gg,\hg)\setminus\{0\}$ be a system of simple roots, 
where $r=\dim\hg$, and denote $H_j:=[X_{-\alpha_j},X_{\alpha_j}]=H_j^*\in\hg$ for $j=1,\dots,r$. 
Then $\{H_1,\dots,H_r\}$ is a basis in $\hg$, 
hence there exist $a_1,\dots,a_r\in\RR$ such that 
$$A=a_1H_1+\cdots+a_rH_r.$$
It follows by \cite[Corollary 2.68]{Kn02} that after a suitable re-labeling of roots we may arrange that   
$a_1,\dots,a_r\in[0,\infty)$. 
Then we define 
\begin{equation}\label{finite_dim_proof_eq1}
Y:=\sqrt{a_1}X_{\alpha_1}+\cdots+\sqrt{a_r}X_{\alpha_r}.
\end{equation}
Since $\alpha_1,\dots,\alpha_r$ are simple roots, it follows by \cite[Lemma 2.51]{Kn02} 
that for $j\ne k$ we have $-\alpha_j+\alpha_k\not\in\Delta(\gg,\hg)\setminus\{0\}$, 
hence by \cite[Proposition 2.5(c)]{Kn02} 
we obtain $[\gg^{-\alpha_j},\gg^{\alpha_k}]\subseteq\gg^{-\alpha_j+\alpha_k}=\{0\}$, 
and in particular $[X_{-\alpha_j},X_{\alpha_k}]=0$. 
On the other hand $[X_{-\alpha_j},X_{\alpha_j}]=H_j$ for $j=1,\dots,r$, 
hence we obtain 
$$[Y^*,Y]=\sum_{j,k=1}^r\sqrt{a_ja_k} [X_{-\alpha_j},X_{\alpha_k}]=
\sum_{j=1}^ra_jH_j=A$$
and this concludes the proof. 
\end{proof}

Let us note the following easy consequence, which for the finite-dimensional classical complex Lie algebras 
is actually weaker than the information provided by Theorem~\ref{Gordon}. 

\begin{corollary}\label{Oberwolfach}
If $\gg$ is a finite-dimensional complex semisimple Lie algebra, 
then for every $A\in\gg$ there exist $X_1,X_2,Y_1,Y_2\in\gg$ for which 
$A=[X_1,X_2]+[Y_1,Y_2]$. 
\end{corollary}

\begin{proof}
Write $A=A_1+iA_2$ with $A_j^*=A_j$ for $j=1,2$, 
and then use Proposition~\ref{finite_dim} for $A_1$ and $A_2$.  
\end{proof}

\begin{remark}\label{reason}
\normalfont
In the proof of Proposition~\ref{finite_dim} we chose $\{\alpha_1,\dots,\alpha_r\}$ 
to be a system of simple roots for the sake of simplicity. 
However, it is easily noticed that the only properties required 
from $\{\alpha_1,\dots,\alpha_r\}$ are the following ones: 
\begin{enumerate}
\item\label{reason_item1} The roots $\alpha_1,\dots,\alpha_r$ are positive with respect to some ordering 
on the self-adjoint part of the Cartan subalgebra under consideration. 
\item\label{reason_item2}  If $j\ne k$, then $\alpha_k-\alpha_j\not\in\Delta(\gg,\hg)\setminus\{0\}$. 
\item\label{reason_item3} The Cartan subalgebra is spanned by the co-roots of $\alpha_1,\dots,\alpha_r$. 
\end{enumerate}
In particular, it is not necessary for $\{\alpha_1,\dots,\alpha_r\}$ to span the whole set of positive roots. 
For instance, if $\gg=\so(2n+1,\mathbb C)$ (the simple Lie algebra of type $B_n$) 
with the root system denoted as usual by $\{e_i\pm e_j\mid 1\le i<j\le n\}\cup\{\pm e_i\mid 1\le i\le n\}$ 
(\cite[Appendix C]{Kn02}), 
then the subset $\{e_i\mid 1\le i\le n\}$ satisfies the above conditions \eqref{reason_item1}--\eqref{reason_item3}
although it fails to be system of simple roots. 
Similarly, if $\gg=\sp(n,\mathbb C)$ (the simple Lie algebra of type $C_n$) 
with the root system 
$\{e_i\pm e_j\mid 1\le i<j\le n\}\cup\{\pm 2e_i\mid 1\le i\le n\}$ 
(\cite[Appendix C]{Kn02}), 
then the subset $\{ 2e_i\mid 1\le i\le n\}$ also satisfies the above conditions \eqref{reason_item1}--\eqref{reason_item3}
although it fails to be a system of simple roots. 
An infinite-dimensional version of the latter example was used in the proof of Theorem~\ref{th_C}. 
\end{remark}

\begin{example}\label{finite_dim_ex}
\normalfont
We wish to show here that by specializing the above construction from the proof of Proposition~\ref{finite_dim} for  
$$\gg=\slalg(r+1,\mathbb C):=\{X\in M_{r+1}(\mathbb C)\mid\Tr X=0\}$$ 
one obtains precisely the matrix $Y$ from Example~\ref{Fan_ex}. 
It follows by \cite[Cor. 6.22]{Kn02} that we may assume that the Cartan involution $X\mapsto X^*$ 
is defined by conjugate transpose. 
By using the spectral theorem we may also assume that $A=A^*\in\slalg(r+1,\mathbb C)$ is a diagonal matrix, 
say $A=\diag(c_1,\dots,c_{r+1})$, where $c_1,\cdots,c_{r+1}\in\RR$ and $c_1+\cdots+c_{r+1}=0$. 
Then after a suitable permutation of the vectors in $\mathbb C^n$ 
(or equivalently, of the rows and columns of $A$) we may assume $c_1\ge\cdots\ge c_{r+1}$, 
and then by Lemma~\ref{Fang_fin} we have $a_j:=c_1+\cdots+c_j\ge 0$ for 
$j=1,\dots,r+1$. 

A maximal abelian self-adjoint subalgebra of $\slalg(r+1,\mathbb C)$ containing $A$ 
is given by the set $\hg$ of all diagonal matrices in $\slalg(r+1,\mathbb C)$. 
Let $E_{jk}\in M_{r+1}(\mathbb C)$ be the matrix whose entry on the position $(j,k)$ is equal to $1$ 
and the other entries are equal to~$0$. 
We have $\Delta(\gg,\hg)=\{\alpha_{jk}\mid j,k\in\{1,\dots,n+1\},\ j\ne k\}$, 
where 
$$\alpha_{jk}\colon\hg\to\mathbb C,\quad \alpha_{j,k}(H)=\Tr((E_{jj}-E_{kk})H)$$
and the corresponding root space is $\gg^{\alpha_{j,k}}=\mathbb C E_{jk}$. 
We will choose $X_{\alpha_{jk}}:=E_{jk}$ and then $X_{\alpha_{jk}}^*=E_{kj}=X_{\alpha_{kk}}=X_{-\alpha_{jk}}$, 
and moreover $[X_{\alpha_{jk}},X_{-\alpha_{jk}}]=E_{jj}-E_{kk}$.
 
A system of simple roots is $\{\alpha_{j,j+1}\mid 1\le j\le r\}$   
and then 
$$\{H_j:=E_{jj}-E_{j+1,j+1}\mid 1\le j\le r\}$$ 
is the corresponding basis in $\hg$. 
The matrix $A=\diag(c_1,\dots,c_{r+1})$ with $c_1+\cdots+c_r+c_{r+1}=0$ can be written as 
$$\begin{aligned}
A
=&c_1E_{11}+\cdots+c_rE_{rr}+c_{r+1}E_{r+1,r+1} \\
=&\sum_{j=1}^r c_jE_{jj}-(\sum_{j=1}^r c_r)E_{r+1,r+1} \\
=&c_1(E_{11}-E_{22})+(c_1+c_2)(E_{22}-E_{33})+\cdots \\
&+(c_1+\cdots+c_r)(E_{rr}-E_{r+1,r+1}) \\
=&\sum_{j=1}^r (\underbrace{c_1+\cdots+c_j}_{\hskip25pt=\,a_j\,\ge \,0}) H_j 
%\\=&\sum_{j=1}^r a_j H_j 
\end{aligned}$$
by using the notation introduced above. 
In the present setting, the equation \eqref{finite_dim_proof_eq1} from the proof of Proposition~\ref{finite_dim} 
specializes as 
\begin{equation}\label{finite_dim_ex_eq1}
Y=\sqrt{a_1}E_{12}+\cdots+\sqrt{a_r}E_{r,r+1}
\end{equation}
which is just the matrix from \eqref{Fan_ex_eq1}. 
Then we have 
$[Y^*,Y]=A$
by the direct computation mentioned in the end of Example~\ref{Fan_ex}, 
which is actually a specialization of the reasoning from the proof of the above Proposition~\ref{finite_dim}.
\end{example}

\begin{remark}\label{secret}
\normalfont
The construction of the operator $Y$ in the proof of Theorem~\ref{th_C} 
is inspired by the method of proof of Proposition~\ref{finite_dim}, 
which used the root space decomposition of a complex semisimple Lie algebra; 
particularly note the similarities between the formulas 
\eqref{Fan_ex_eq1}, \eqref{finite_dim_proof_eq1}, \eqref{finite_dim_ex_eq1}, and  \eqref{th_C_proof_eq1}. 
We recall that the root system of the infinite-dimensional classical Lie algebra $\sp_\infty(\Hc)$ 
can be found for instance in \cite[page 41, Prop. 4C]{dlH72}, 
and the above rank-two operators $X_j$ are root vectors corresponding 
to the simple root systems indicated in \cite[Eq. (5.8)]{NRW01}; 
see also \cite{Sch60}. 
\end{remark}

%%%%%%%%%%%%%%%%%%%%%%%%%%%%%%
%%%%%%%%%%%%%%%%%%%%%%

\end{document}